\documentclass{amsart}
\usepackage{amsmath}
\usepackage{stackrel}
\usepackage{amssymb, MnSymbol}
\usepackage{color}
\usepackage{amscd}
\input xy
\xyoption{all}
\newtheorem{Lemma}{Lemma}[section]
\newtheorem{remark}[Lemma]{Remark}

\newtheorem{theorem}[Lemma]{Theorem}
\newtheorem{lemma}[Lemma]{Lemma}
\newtheorem{proposition}[Lemma]{Proposition}
\newtheorem{corollary}[Lemma]{Corollary}

\newtheorem{example}[Lemma]{Example}

\newcommand{\Cal}[1]{{\mathcal #1}}
\newcommand{\End}{\operatorname{End}}

\newcommand{\Hom}{\operatorname{Hom}}

\newcommand{\Mod}{\operatorname{Mod-\!}}

\DeclareMathOperator{\Add}{Add}

\newcommand{\soc}{\mbox{\rm soc}}

\newcommand{\cmat}{\left(\begin{array}}
\newcommand{\fmat}{\end{array}\right)}

\newcommand{\N}{\mathbb N}
\newcommand{\Z}{\mathbb{Z}}

\newcommand{\Q}{\mathbb{Q}}

\begin{document}
   \title{Covering classes and uniserial modules}
     \author[Alberto Facchini]{Alberto Facchini}
\address{Dipartimento di Matematica, Universit\`a di Padova, 35121 Padova, Italy}
 \email{facchini@math.unipd.it}
\thanks{The first author was partially supported by Ministero dell'Istruzione, dell'Universit\`a e della Ricerca (Progetto di ricerca di rilevante interesse nazionale ``Categories, Algebras: Ring-Theoretical and Homological Approaches (CARTHA)'') and Dipartimento di Matematica ``Tullio Levi-Civita'' of Universit\`a di Padova (Research program DOR1828909 ``Anelli e categorie di moduli''). The second author was supported by a grant from  IPM. The third author was supported by the grant 
GA\v CR 17-23112S.\\
Parts of this paper were written while the first author was visiting the IPM or the
second author was visiting Charles University or the University of Padova. We all would like to thank these three institutions for their hospitality.}

 \author[Zahra Nazemian]{Zahra Nazemian}
\address{School of Mathematics, Institute for Research in Fundamental Sciences (IPM),
	P. O. Box: 19395-5746, Tehran, Iran}
 \email{z\_nazemian@yahoo.com}
\author[Pavel P\v r\'\i hoda]{Pavel P\v r\'\i hoda}
\address{Charles University in Prague, Faculty of Mathematics and Physics, Department of Algebra, Sokolovsk\'a 83, 186 75 Praha 8, Czech Republic}
 \email{prihoda@karlin.mff.cuni.cz}


      \begin{abstract}  We apply minimal weakly
generating sets to study the existence of $\Add(U_R)$-covers for a uniserial module $U_R$. 
 If $U_R$ is a uniserial right module over a ring $R$, then $S:=\End (U_R)$ has at most two maximal (right, left, two-sided) ideals:
 one is 
the set $I$ of all endomorphisms that are not injective, and the other is the set $K $ of all endomorphisms of $U_R$ that are not surjective. 
We prove that if $U_R$ is either finitely generated, or artinian, or  $I \subset K$, then the class $\Add(U_R)$ is covering if and only if it is closed under direct limit. Moreover, we study endomorphism rings of artinian uniserial
modules giving several examples.
\end{abstract}

    \maketitle

\section{Introduction}

For a right module $M_R$ over a ring $R$ with identity, let $\Add(M_R)$ denote the class of all right $R$-modules that are isomorphic to a direct summand of a direct sum $M_R^{(I)}$ of copies of $M_R$. 
The class $\Add(M_R)$ is a precovering class for every module $M_R$, because it suffices to take, as an  $\Add(M_R)$-precover of a module $A_R$, the canonical mapping $M^{(\Hom(M_R,A_R))}\to A_R$, $(m_f)_{f}\mapsto \sum_ff(m_f)$. If a precovering class is closed under direct limit, then it is covering. That is, every module $A$ has an $\Add(M_R)$-cover. It is not known whether the converse holds (this is called the {\em limit conjecture} or {\em Enochs' conjecture}). 

Assume that $M_R$ is finitely generated and $\Add(M_R)$ consists only of modules isomorphic to direct sums of copies of $M$, e.g., $S:=\End(M_R)$ local \cite[Corollary~2.55]{libro}.  If $A_R$ is an arbitrary $R$-module, the problem of determining
whether $A_R$ has an $\Add(M_R)$-cover is determined by the structure of 
the $S$-module $\Hom(M_R,A_R)$. Indeed, it is easy to see that every generating set 
$X$ of $\Hom(M_R,A_R)$ defines a precover $f_X \colon M_R^{(X)} \to A_R$, 
$(m_x)_{x} \mapsto \sum_xx(m_x)$ and that, essentially, every $\Add(M_R)$-precover of $A_R$
arises in this way. Moreover, $f_X$ is a cover if and only if $\{\,g \in \End(M_R^{(X)}) \mid f_X g = f_X\,\}
\subseteq J(\End(M_R^{(X)}))$. Verifying this condition for a given generating set $X$
can be difficult, and sometimes it is convenient to consider the weaker condition that $X$ is a minimal generating set, because if $f_X$ is a cover, then $X$ is a minimal generating set for the right $S$-module $\Hom(M_R,A_R)$. So, if $\Hom(M_R,A_R)_S$ has no minimal generating set, 
then $A_R$ does not have an $\Add(M_R)$-cover. Modules with or without a minimal generating 
set have been studied in the literature, for example, modules over Dedekind domains with minimal generating 
sets were characterized in \cite{HR}.
When $M_R$ is not finitely generated, we must also consider minimal {\em weakly} generating 
sets of $\Hom(M_R,A_R)_S$. We introduce this notion in Section \ref{s2}.  

In order to illustrate the use of minimal generating sets, let us 
show how it is possible to prove  a very particular instance of \cite[Theorem~4.4]{procedinglondonlidia}.
Consider  
$M_R$ finitely presented with local endomorphism ring $S$ and suppose 
$J(S)$ not $T$-nilpotent, that is, there exists a sequence 
$$M\stackrel {f_1}{\to} M\stackrel{f_2}{\to} M \stackrel{f_3}{\to} \cdots $$
with $f_1,f_2,\dots \in J(S)$ such that the direct limit $L$ of the sequence is non-zero. 
Since $M$ is finitely presented, every element of $\Hom(M_R,L_R)$ factors through 
a colimit injection, and $\Hom(M_R,L_R) = \Hom(M_R,L_R)J(S)$ follows. By \cite[Lemma~3.1]{HR},
$\Hom(M_R,L_R)_S$ has no minimal generating set and hence $L_R$ does not have 
an $\Add(M_R)$-cover.

Conversely, assume that $M_R$ is a finitely presented module with local endomorphism ring $S$ whose Jacobson radical $J(S)$ is 
$T$-nilpotent.  Let $X$ be a subset of $\Hom(M_R,A_R)$
whose image in $\Hom(M_R,A_R)_S/\Hom(M_R,A_R)J(S)$ is a free generating set over $S/J(S)$.
Then it is possible to prove that $X$ is a minimal generating set for  
$\Hom(M_R,A_R)_S$. Moreover, $\{\,g \in \End(M_R^{(X)}) \mid f_X g = f_X\,\} \subseteq J(\End(M_R^{(X)}))$ because $J(S)$ is $T$-nilpotent \cite[Theorem~1]{Zelmanowitz}. So in this case $\Add(M_R)$ is covering. 

In this work, we study when $\Add(U_R)$-covers exist (or do not exist) for a uniserial module $U_R$.
The class $\Add(U_R)$ for a uniserial module $U_R$ was completely determined by the third author in \cite{Pavel}.
The class $\Add(U_R)$ is either {\em trivial}, that is, all its elements are isomorphic to $U_R^{(X)} $ for some suitable set $X$,  or there exists a uniserial module $V_R$ such that every element of the class is isomorphic to $U_R^{(X)} \oplus V_R^{(Y)}$ for suitable index sets  $X$ and $Y$. 
 
Unluckily, knowing $\Hom(U_R,A_R)_S$ is usually 
not sufficient to determine the existence of $\Add(U_R)$-cover of $A_R$. 
We prove that $A_R$ does not have an 
$\Add(U_R)$-cover when $A_R$ is either a factor of $U$, or 
a union of submodules isomorphic to $U$, or a product of copies of $U$ 
provided $U$ satisfies some additional conditions.

A related problem we consider in this paper is whether it is true that, for a uniserial module $U_R$, 
$\Add(U_R)$ is covering if and only if it is closed under direct limit. We verify Enochs' conjecture in three cases: $U_R$ finitely generated, $U_R$ artinian, and  $I \subset K$, that is, every surjective endomorphism of $U_R$ is an automorphism of $U_R$ but there exists an injective endomorphism of $U_R$ that is not an automorphism. See Sections \ref{3} and \ref{5}. 

Examples of uniserial modules $U$ for which the class $\Add(U)$ is covering are given by the uniserial modules $U$ for which 
every 
family of modules $\{\,U_i\mid {i\in I}\,\}\  (U_i = U$ for every $i\in I) $ is  locally  $T$-nilpotent (Section \ref{s2}). 
In the case of artinian uniserial modules,  there is a pair of ordinal numbers that completely characterizes a number of properties of 
$U$, for example when every family of modules $\{U_i\mid {i\in I}\}\  (U_i = U) $ is  locally  $T$-nilpotent,  or when 
 $U$ is self small. See Section \ref{4}. In Section \ref{6}, some examples of artinian uniserial modules and their endomorphism rings are given. 
 
 \smallskip
  
In the whole paper, $U_R$ denotes a uniserial right $R$-module and  $S=\End(U_R)$ is its  endomorphism ring.

\section{Preliminaries. Weakly generating sets.} \label{s2}

For any uniserial $R$-module $U_R\ne0$, the endomorphism ring $S:=\End(U_R)$ exactly has either one or two maximal (right, left, two-sided) ideals \cite{Facchinitransaction}. More precisely, $S$ always contains the two completely prime ideals $I:=\{\,f\in S\mid f$ non-monic$\,\}$ and $K:=\{\,f\in S\mid f$ non-epi$\,\}$. If $I$ and $K$ are comparable, then $S$ is local and its maximal ideal is the larger between $I$ and $K$. In this case, we say that $U_R$ is {\em of type $1$}. If $I$ and $K$ are not comparable, that is, if $I\nsubseteq K$ and $K\nsubseteq I$, then $I$ and $K$ are the two maximal (right, left, two-sided) ideals of $S$.
In this case, we say that $U_R$ is {\em of type $2$}. When $U_R$ is artinian uniserial, then $S$ is always local with maximal ideal~$I$.

\medskip

Recall that a family of modules $\{M_i\mid {i\in I}\}$ is said to be {\em locally semi-$T$-nilpotent} if, for each sequence of 
non-isomorphisms
 $M_{i_1}\stackrel{f_1} {\longrightarrow}  M_{i_2}\stackrel{f_2} {\longrightarrow} 
 M_{i_3} \stackrel{f_3} {\longrightarrow} \cdots $ with pairwise different indices $(i_n)_{n\in\N}$ from $I$, and 
each element $x \in M_{i_1}$, there exists $m = m(x) \in\N$ such that $f_mf_{m-1}\dots f_1(x) = 0$ \cite{lidia}. If the 
same condition is also satisfied when we allow repetitions in the sequence of indices $\{\,i_n \mid n\in\N\,\}$ involved, then 
the family $\{\,M_i\mid {i\in I}\,\}$ is said to be {\em locally $T$-nilpotent}. In most cases in this paper, all the modules $\{\,M_i\mid i
\in I\,\} $ in the family will be equal to a unique module $M$, in which case there is no difference
 between being locally $T$-nilpotent or locally semi-$T$-nilpotent. Also, for any infinite set $I$, the family $\{\,M_i\mid i\in I\,\}$,  where $M_i=M$ for every $i\in I$, is locally $T$-nilpotent if and only if the family $\{\,M_n\mid n\in \N\,\}$, where $M_n=M$ for every $n\in\N$, is locally $T$-nilpotent.

 If $U$ is a finitely presented uniserial module, then the class $\Add(U)$ is covering if and only if $\Add(U)$ is closed under direct limit 
\cite [Theorem 4.4]{procedinglondonlidia}. From the result in \cite {lidia}, we get a class of uniserial modules $U$ for which $\Add(U)$ is covering and closed under direct limit:

\begin{lemma}\label{localy}
Let $U$ be a uniserial module for which the family   $\{U_i\mid {i\in I}\}$, with $U_i=U$ for all $i\ge1$,
is locally $T$-nilpotent. Then:

{\rm (1)} $\End(U)$ is local  with 
maximal ideal $I$, the set of all endomorphisms of $U$ that are not monic.

{\rm (2)} $\Add(U)$ is covering. 

{\rm (3)}  $\Add(U)$ is closed under direct limit.\end{lemma}

\begin{proof} (1) If $\End(U)$ is not local with maximal ideal $I$, then there exists $f\colon U\to U$ monic but not epi. Then the sequence $U\stackrel{f} {\longrightarrow}  U\stackrel{f} {\longrightarrow}  U\stackrel{f} {\longrightarrow}   \cdots $ shows that the family is not locally $T$-nilpotent.  This proves that every monic is epi, that is, $I$ is the maximal ideal of $S$. 

(2) follows from \cite[Propositions~4.2 and~4.1]{procedinglondonlidia}.

(3) 
Set  $M := \oplus _{i = 1} ^{\infty}  U_i$, where $U_i = U$. Then 
$M$  satisfies Condition 2, and so 3, of \cite [Theorem 1.1]{lidia}. 
Thus, by \cite [Theorem  1.4]{lidia}, $\Add(M)$, and hence $\Add(U)$, are closed under direct limit. 
\end{proof}

\smallskip

Let $N_R$ and $M_R$ be two $R$-modules. 
A family $\{\,s_x\mid x\in X\,\}$ of elements of  $S = \End (N_R)$ is said to be {\em summable} if, for every $n\in N_R$,  $s_x(n) \neq 0$ only for finitely many $x \in X$. Thus a family $\{\,s_x\mid x\in X\,\}$ of elements of  $S = \End (N_R)$ determines  a morphism $\varphi\colon N_R\to N_R^X$, and the family $\{\,s_x\mid x\in X\,\}$ is summable if and only if the image of $\varphi$ is contained in $N_R^{(X)}$.
We say that $X \subseteq \Hom(N, M)$  {\em weakly generates} $\Hom(N, M)$ if 
any element $g \in \Hom(N, M)$  can be written as $\sum _{x\in X} x s_x $ for some 
summable family $\{\,s_x\mid x \in X\,\}$ of elements in $S$ (in this case we say that $g$ is {\em weakly generated }by $X$).
 Also, we say that $X$ is a {\em minimal weakly generating set}  for  $\Hom(N, M)$ 
 if no element $x$ of $X$ can be weakly generated by $X\setminus\{x\}$.  

Clearly, if  $X$  generates $\Hom(N, M)$,
 then it weakly generates $\Hom(N, M)$.
Moreover, if $X \subseteq Y$ and $X$ weakly generates  $\Hom(N, M)$, then $Y$  
weakly generates  $\Hom(N, M)$ too. 
We say that {\em $\Add(N)$ is trivial} if every module in $\Add(N)$ is isomorphic to a direct sum of copies of $N$ \cite{Pavel}. For instance, $\Add(U)$ is trivial for every uniserial module $U$ of type $1$ \cite[Theorem~2.52]{libro}.

\begin{theorem} \label{mgset}
Let $M$ and $N$ be two right $R$-modules. 

{\rm (1)} If   $\Add(N)$ is trivial and $X $ is a weakly generating set for $\Hom (N, M)$, then the map 
$f \colon N^{(X)} \to M$ 
which sends $(n_x)_{x \in X}$ to $\sum_{x \in X} x(n_x)$ is an $\Add(N)$-precover of $M$. 

{\rm (2)} If $f\colon N^{(X)} \to M$ is an $\Add(N)$-precover for $M$ and the mappings 
 $i_x \colon N \to N^{(X)}$, $x\in X$, are the canonical embeddings, 
then $\{ fi_x \colon N \to M \mid x\in X\}$ is a  weakly generating set for $\Hom(N, M)$. 

{\rm (3)} If $f\colon N^{(X)} \to M$ is an $\Add(N)$-cover, then $\{ fi_x \colon  N \to M \mid x\in X\}$ is a minimal weakly generating set of  $\Hom(N, M)$. 
\end{theorem}

\begin{proof}
(1) Let $Y$ be an index set  with a homomorphism $h\colon N^{(Y)} \to  M.$
For each $y \in Y$,  write  $h \iota_y\colon  N \to  M$  in the
form $\sum_{x \in X} xg_{x,y}$, where the elements $g_{x,y} \in S$  form a summable family. 
 So for each $y$ there is a map $h_y \colon  N_y = N \to N^{(X)}$ such that 
$h_y (n) = (g_{x,y} (n))_{x \in X}$. By the universal property of direct sum for 
$N^{(Y)}$, there exists a  homomorphism $\beta\colon  N^{(Y)} \to  N^{(X)}$
such that $\beta \iota_y = h_y$ for every $y \in Y$. 
Then it is easy to check  that $f  \beta = h$,  hence $f$ is an $\Add(N)$-precover of $M$. 

(2) If $g \colon N \to M$, then there exists $h\colon  N \to N ^{(X)} $ such that $fh= g$. 
Consider the canonical projections $p_x\colon N^{(X)} \to N$, $x\in X$, and 
note that $\{\, p_xh\mid x \in X\,\}$ is a summable family of $\End(N)$.
Moreover, $g(n) = \sum_{x \in X} f i_x p_x h (n)$ for every $n \in N$, so $g = \sum_{x\in X} (fi_x)(p_xh)$. 

(3)  Set $Y := X \setminus \{ x \}$ and assume
$fi_x = \sum _{y \in Y} fi_y g_y $, where $\{\,g_y\mid y \in Y\,\}$ is a summable family in $\End(N)$.  
Define $\varphi \colon  N^{(X)} \to N^{(X)}$ such that $\varphi i_x = \sum_{y\in Y} i_y g_y  $
and $\varphi i_y =  i_y $, for every $y \in Y$. Then $\varphi$ is neither a monomorphism 
nor an epimorphism but $f \varphi = f $, which is a contradiction. 
\end{proof}

\begin{remark}\label{remark1.3}{\rm Recall that a module $N_R$ is said to be {\em self-small} if for every set $X$ and every morphism $\varphi\colon N_R\to N_R^{(X)}$ there exists a finite subset $F$ of $X$ such that $\varphi(N_R)\subseteq N_R^{(F)}$. 
When the class $\Add(N)$ is trivial and $N$ is self-small (for example, $N$  finitely generated or, more generally, small),  then  any
$\Add(N)$-precover of a module  $M$  is of the form 
$f \colon  N^{(X)} \to M$, where $X$  
is a generating set of $\Hom(N_R,M_R)$ and  $f$  is defined as in  Theorem~\ref{mgset}(1). 
To see this, notice that if $\Add(N)$ is trivial, then any $\Add(N)$-precover of $M$ is of the form 
 $ g\colon N^{(I) } \to  M $. 
Therefore consider the homomorphisms $ g_i\colon  N \to M,\ i \in I,$  
given by $g_i := g\iota_{i}$. Then $\{g_i \mid  i \in I\}$
is a generating set for $\Hom(N,M)$. Indeed,  let  $h\colon N \to M$ be such
that $h \not \in \sum_{i \in I} g_iS$. Since $N$ is self-small, $h$ 
does not  not factor through $g$, which is impossible if $g$ is an $\Add(N)$-precover of $M$. }
\end {remark}

Therefore Theorem~\ref {mgset}(2) implies:

\begin{corollary}
Suppose $N$ self-small and $\Add(N)$ trivial. If a module $M$ has an $\Add(N)$-cover, then 
the $S$-module $\Hom(N,M)$ has a minimal generating set. In particular, if $\Add(N)$ is covering,
then for every module $M_R$ the $S$-module $\Hom(N_R,M_R)$ has a minimal generating set.
\end{corollary} 

\begin{lemma} 
Let
$U_R$ be a self-small uniserial right $R$-module  with $\Add(U)$ covering. If  
$S = \End (U_R) $   is
a left chain domain, then  $S$ is a division ring and  $\Add(U)$ is closed under direct limit.  
\end{lemma}

\begin{proof}
If $\Add(U)$ is covering, then the right  $S$-module $M = \Hom (U, U^I) = S^I$ has a minimal generating set for every
index set $I$. 
Let $X$
be a minimal generating set for $M_S$. Let us show that $X$ is a free generating set for $M_S$. 
If  not, then  there exist $s_1, \dots , s_n \in S$ and $x_1, \dots, x_n \in X$ such that 
$ x_1 s_1+ \cdots  +  x_n s_n = 0$ and the elements $x_is_i $ are non-zero.
 Since $S$ is left chain, we can assume that  
$s_i =   s_i' s_1$ for every $i > 1$. 
Now since $S$ is a domain and $M=S^I$ is torsion-free, we see that  $x_1$  would be generated by the other elements in $X$, which is a contradiction.
So any product of copies of $S$ is a free right $S$-module. By 
\cite [Theorem 3.3] {Chase},  $S$ is right perfect and, 
by Lemma~\ref{localy},  $\Add(U)$ is closed under direct limit. 
\end{proof}

Let $U_R$ be a uniserial module over an arbitrarily fixed ring $R$. For every right $R$-module $N_R$, define a group topology on the abelian additive group $\Hom(U_R,N_R)$ taking as a basis of neighbourhoods of $0$ the subgroups $H_u:=\{\,f\in \Hom(U_R,N_R)\mid f(u)=0\,\}$ of $\Hom(U_R,N_R)$, where $u$ ranges in $U$. This topology on $\Hom(U_R,N_R)$ is usually called the {\em finite topology.} 
Notice that the  subgroups $H_u$ ($u\in U)$ are linearly ordered by set inclusion. It is easy to check that:

\begin{Lemma} For every right $R$-module $N_R$, $\Hom(U_R,N_R)$ is a complete Hausdorff topological group.\end{Lemma}


In particular, for $N_R=U_R$, we get a group topology on the ring $S=\End(U_R)$. Now $_SU_R$ is a bimodule and $H_u=l.ann_S(u)$ is a left ideal of $S$ for every $u\in U$, so that the finite topology turns out to be a left linear topology on $S$. (In the special case of $U_R$ artinian, all the subgroups $H_u$ are $S$-submodules of the right $S$-module $\Hom(U_R,N_R)$, but we will not need this further hypothesis $U_R$ artinian in the following.)

\begin{Lemma} For every right $R$-module $N_R$, $\Hom(_SU_R,N_R)_S$
is a topological right module over the left linearly topological ring $S$.\end{Lemma}

\begin{proof} We must show that the three axioms $NM\,3$-$NM\,5$ in \cite[page 144]{St} hold. Let $H_u$ ($u\in U)$ be the subgroups of $\Hom(U_R,N_R)$ and $H'_u$ ($u\in U)$ be the left ideals in $S$.
Then
$NM\,3$ holds, because for each $f\in\Hom(U_R,N_R)$ and each $H_u$, we get that $H'_u$ has the property that $fH'_u\subseteq H_u$. $NM\,4$ holds, because for each $H_u$ and each $s\in S$, we have that $H_{su}$ is such that $H_{su}s \subseteq H_u$. Finally, $NM\,5$ holds, because, for each $H_u$ ($u\in U$), we have $H_0H'_u\subseteq H_u$.\end{proof}

Trivially, if the right $S$-module $M_S:=\Hom(U_R,N_R)$ has a minimal weakly generating set $X$ and $M'_S=\sum_{x\in X}xS$ is the $S$-submodule of $M_S$ generated by $X$, then 
$M_S$ is the completion of its topological submodule $M'_S$ with the induced topology and $X$ is a minimal generating set for $M'_S$. (To prove that $M'_S$ is dense in $M_S$, fix $f\in M_S$ and $u\in U$. One must show that $(f+H_u)\cap M'_S\ne\emptyset$. Now $f=  \sum _{x\in X} x s_x $ for some 
 summable family of elements of $S$. Hence there exists a finite subset $F$ of $X$ such that $s_x(u)=0$ for every $x\in X\setminus F$. Then $\sum _{x\in X\setminus F} x s_x \in H_u$.)
 
 \medskip
 
Notice that a family $X\subseteq S$ is summable if and only if the set $X\setminus U$ is finite for every neighourhood $U$ of $0$ in $S$. Moreover, we have the following proposition, easy to prove. It improves Lemma~\ref{localy}.

\begin{proposition} The following conditions are equivalent for a uniserial module 
$U_R$ with endomorphism ring $S$:

{\rm (1)} The family $\{\,U _i\mid{i \geq 1 }\,\}$, with $U_i=U$ for all $i\ge1$,
is locally $T$-nilpotent.

{\rm (2)} The ring $S$ is local with maximal ideal $I$, the set of all endomorphisms of $U$ that are not monic, and, for every strictly descending chain $Ss_1\supset Ss_2\supset Ss_3\supset \dots$ of principal left ideals of $S$, the family $\{\,s_i\mid n\ge 1\,\}$ is summable.
\end{proposition}

\begin{Lemma}\label{minmax} Let $M_R,N_R$ be right $R$-modules, and $S:=\End(N_R)$. If\linebreak $\Hom(N_R,M_R)_S$ is non-zero and has a minimal weakly generating set as an $S$-module, then $\Hom(N_R,M_R)_S$ has a maximal submodule.  \end{Lemma}

\begin{proof} Suppose that $\Hom(N_R,M_R)_S\ne 0$ has a minimal weakly 
generating set $X$. From $\Hom(N_R,M_R)\ne 0$, we get that $X\ne\emptyset$. Fix $x\in X$. Then the set of all elements of $\Hom(N_R,M_R)_S$ weakly generated by $X\setminus\{x\}$ is a submodule $H$ of $\Hom(N_R,M_R)_S$, which is proper, because $x\notin H$. Hence $\Hom(N_R,M_R)_S/H$ is a non-zero cyclic right $S$-module. Thus $\Hom(N_R,M_R)_S/H$ has a maximal submodule, so $\Hom(N_R,M_R)_S$ has a maximal submodule.\end{proof}

The following elementary lemma is necessary for the sequel. 

\begin{lemma}\label{maximalsubmodule}
Let $S$ be a local ring, and $M$
 be its maximal ideal. Then a module $_SA $  has a maximal submodule if and only if $MA 
\neq A$.
\end{lemma}
\begin{proof} If $MA \neq A$, then $A/MA$ is a non-zero left vector space over the division ring 
$S/M,$ hence has a maximal submodule.
Conversely, if $_SA$ has a maximal submodule $ A'$, then $A/A'$ is a simple left $S$-module,
 hence is a module over  $S/M,$ so $MA \leq A' \neq A$.
 \end{proof}

If $U_1$  and $U_2$  are  uniserial modules, we say that 
$U_1$ and $ U_2$  are in the same {\em epigeny  
class} if there are epimorphisms $ f:U_1 \to U_2$  and $g : U_2 \to U_1$. 
In this case we write $[U_1]_e = [U_2]_e$. 
Recall that for a uniserial module $U$,  $U_e$ is the union of the kernels of all epimorphisms in $S = \End (U)$.
Note that if $U_e$  is non-zero and $V$ is a submodule of $U$,
 then $[U]_e = [U/V]_e$ if and only if $V < U_e$. 
In particular, $U$ and $U/U_e$ are not in the same epigeny class. 
In fact, if $g\colon  U/V \to U$ is an epimorphism, $U_e \leq V$, 
$h \colon  U \to U$ is an epimorphism with non-zero kernel, and $\pi \colon  U\to U/ U_e$ is the canonical projection, then   
$ hg \pi \colon  U \to U$  would be an epimorphism with kernel properly containing $U_e$, which is a contradiction.  

Note that if $N \leq N'$ are submodules of $U$ such that 
$N = \ker (f_N)$  and $N' = \ker(f_{N'})$ for suitable  epimorphisms $f_N ,  f_{N'} \in S$, then
the map $f _{N, N'} \in S$, defined by $f _{N, N'} (u) = f_{N'} (x)$ if $u = f_N (x)$, is the unique map 
with the property that  $f _{N, N'} f_N  = f_{N'}$. 

Set  $T: = \{ N \leq U \mid N=\ker (f_N)$  for some epimorphism $f_N \in S \} $. 
Then we have a direct system $(U_N) _{N  \in T}$, where $U_N = U$ and, for 
$N \leq N'$,   we have $f _{N, N'} \colon  U_N \to U_{N'}$. Then the direct limit of the system is $U/U_{e}$.  

\begin{theorem}\label{Ue}
Let  $U$ be of type $1$ (that is, suppose $S$ local). If $\Add(U)$ is covering, then either $U_e = 0$ or
$U_e = U$.
\end{theorem}
\begin{proof}
Suppose $\Add(U)$ covering and $ 0 \neq U_e < U$.
Then $\Hom(U,U/U_e)$ contains a non-zero epimorphism,  $p$ say.
The kernel of every epimorphism $g\colon U \to U/U_e$ contains $U_e$.
Otherwise, if $\ker (g) \varsubsetneq U_e$, then $[U]_e = [U/\ker (g)]_e = [U/U_e]_e$, which is contradiction.
It follows that every epimorphism of $\Hom(U,U/U_e)$ is in $\Hom(U,U/U_e)f$, where 
$f\in \End(U)$ is an arbitrary epimorphism.

Note that every element of $\Hom (U,U/U_e)$ is a sum of at most two  epimorphisms. Indeed, 
if $h \colon U \to U/U_e$ is not onto, then $h = p + (h-p)$, where $p$ and $h-p$ are onto.
It follows that $\Hom(U,U/U_e) = \Hom(U,U/U_e)I$ (since $I$ contains an epimorphism).
So $\Hom(U,U/U_e)$ does not have a maximal submodule, and consequently $U/U_e$ 
does not have an $\Add(U)$-cover. \end{proof}

\section{Finitely generated uniserial modules and modules of type $1$ with $I \subset K$}\label{3}

Note that a uniserial module $U$ has a maximal submodule if and only if $U$ is finitely generated (= cyclic). 

\begin{proposition}\label{noetheriancase}
Let $U_R$ be a finitely generated uniserial module with maximal submodule $M$. Assume that $S$ is local 
with maximal ideal $K$, the set of all non-epi endomorphisms of $U_R$.   If $J$ is an infinite set of cardinality $|J| \geq |U|$, and the module   $N = (U/M) ^ {(J)}$ has an $\Add(U)$-cover, then $\{\,U_i\mid {i\ge 1}\,\}$ with $U_i = U $ for every $i\ge1$ is locally 
$T$-nilpotent. Therefore  $\Add(U)$ is covering and closed under direct limit. 
\end{proposition}

\begin{proof}
Since $S$ is local, $\Add(U)$ is trivial, and therefore the $\Add(U)$-cover of $N$ is of form 
$U^{(X)}$ for some index set $X$. Let $ \varphi \colon  U^{(X)} \to N$ be the $\Add(U)$-cover of 
$N$. Since there is an epimorphism   $p\colon  U ^{(J)} \to N $, we see that $\varphi$ must be an epimorphism. 
So $X$ can not be finite. 
If  $\{\,U_i\mid i \geq 1\,\} $  is not locally $T$-nilpotent, then  $\{\,U_x\mid x \in X \,\}$ 
is not locally $T$-nilpotent. By \cite [Theorem 1.1]{Pacefic}, there exists 
$h \in \End (U^{(X)}) $ such that $\pi _x h i_y $  is in $K$ for any $x, y \in X$, 
 but $h$ is not in $J (\End (U^{(X)}) $ (here  $i_y$ is the embedding of $U $ onto the $y$-th component of $U^{(X)}$ and 
$\pi_x$ is the canonical projection from $U^{(X)}$ to $U_{x} $). 
There exists $h'\in \End  (U^{(X)})$ such that $1-h'h$ is not invertible. Since $U$ is finitely generated, for each 
$y$ there exist finitely many $t_1,\dots,t_n\in X$ 
 such that  Image$(h i_y) \leq \sum_{m=1}^n i_{t_m} (U)$. Therefore,   for 
each $x$, $\pi_x h'h i_y =  \sum_{m = 1} ^ n  \pi_x h' i_{t_m} \pi _{t_m} h i_y $,
 and so   $\pi_x h'h i_y  \in K$.  Set $k := h' h$.  
Since $U$ is finitely generated, we have 
$  \Hom (U, N) \cong   \Hom (U, U/M) ^{(J)} $, and since   $ \Hom (U, U/M) K = 0$, we get that 
$  \Hom (U, N) K = 0$. 
 Therefore 
$ \varphi k i_x  = \sum _{y\in X} \varphi  i_y \pi _y  k i_x  = 0$. 
So $\varphi (1-k) = \varphi$ and $1-k$ is not invertible, which is a contradiction.
 Now the rest of the proposition  follows from Lemma~\ref{localy}. 
\end{proof}

\begin{lemma}\label{type 2}
If $U$ is a non-zero finitely generated uniserial module and its endomorphism ring is not local, then 
$\Add(U)$ is not a covering class. 
\end{lemma}

\begin{proof}
Let 
$N$ be the maximal submodule of $U$. Suppose that $U/N$ has an $\Add(U)$-cover. 
If $\Add(U)$ is not trivial, then there exists a uniserial non-quasismall  module $V$ such that every element of  
$\Add(U)$ is of form $U^{(X)} \oplus V^{(Y)} $ \cite{Pavel}. A non-quasismall uniserial module is not finitely generated, so 
$\Hom(V, U/N)$ is zero. Therefore, in both cases $\Add(U)$ trivial or non-trivial, the 
$\Add(U)$-cover of $U/N$ is of form $ \varphi \colon U^{(X)} \to U/N$ for a suitable index set $X$. The endomorphism
$S$  has two maximal ideals $I$ and $K$ and $I + K = S$. So $1 = i + k$ for some $i\in I$ and $k \in K$. 
For every $x\in X$, let $i_x$ denote the canonical mapping from 
$U$ to $U^{(X)}$. Consider the map $h \colon  U^{(X)} \to U^{(X)} $ such that $h (i_x (u)) = i_x (i(u))$ for every $u\in U$ and every $x\in X$.
Since  $\Hom (U, U/N) K = 0$, we see that $\varphi i_x = \varphi i_x i$, that is $\varphi h = \varphi$. 
 But $h$ is not a monomorphism, because $i$ is not, and this is a contradiction. 
\end{proof}

We are ready for the first two of the main results of this paper: Corollary~\ref{f.g} and Theorem~\ref{3.4}.

\begin{corollary}\label{f.g}
Let $U_R$ be a finitely generated uniserial module. If $\Add(U)$ is covering, then it is closed under direct limit.   
\end{corollary}

\begin{proof}
By Lemma~\ref {type 2}, the uniserial module $U_R$ is of type $1$. Therefore either every monic endomorphism of $U_R$ is an isomorphism or every onto endomorphism of $U_R$ is an isomorphism. In both cases, we can apply \cite[Theorem~1.1(i)]{Pavel}, getting that $\Add(U)$ is trivial. Let $M$ be the maximal submodule of $U$. If $U_e \neq 0$, then there exists an epimorphism  $f\in S$ that is not monic. Let us show that $\Hom (U, U/M) = \Hom (U, U/M)f$. If  $g \in \Hom (U, U/M)$  is non-zero, then  $\ker (g) = M$. In particular, $\ker (f) \subseteq \ker (g)$ and, since $f$ is onto,
there exists a $\widetilde{g} \in \Hom(U,U/M)$ such that $g = \widetilde{g} f \in \Hom(U,U/M) f$.  
Therefore $\Hom (U, U/M) =  \Hom (U, U/M)I$, where $I$ is the maximal ideal of
$\End(U)$,  that is, by Lemma \ref{maximalsubmodule}, 
  $\Hom(U, U/M)$  is without maximal submodules, and so without minimal generating set, which contradicts the assumption that $\Add(U)$ is covering (Remark \ref{remark1.3}). 
So $U_e = 0$, and thus we can apply 
Proposition \ref{noetheriancase}. In particular,  $\Add(U)$ is closed under direct limit.  
\end{proof}

\begin{theorem}\label{3.4}
Let $U$ be a uniserial module of type $1$ such that $S = \End(U)$ contains a monomorphism 
which is not onto. Then there exists a uniserial module $V$ without an 
$\Add(U)$-cover.  
\end{theorem} 

\begin{proof}
We claim there exists a limit ordinal $\beta$ and a direct system $(U_{\alpha})_{\alpha < \beta}$
such that
\begin{enumerate}
\item[1.] Each module $U_{\alpha}$ is isomorphic to $U$.
\item[2.] For each $\gamma < \alpha < \beta$ the homomorphism $f_{\alpha,\gamma} \colon U_{\gamma} \to 
U_\alpha$ is injective and not surjective. 
\item[3.] If $\alpha$ is limit and $\alpha < \beta$, then $U_{\alpha}$ is a direct limit of 
the system $(U_{\gamma})_{\gamma < \alpha}$.
\item[4.] The direct limit of the system $(U_{\alpha})_{\alpha < \beta}$ is not isomorphic to $U$.
\end{enumerate}

The construction of the direct system begins with $U_{0} := U$. If $U_{\alpha}$ and $ f_{\gamma,\delta} \colon 
U_{\delta}\to U_{\gamma}$ have been defined for every $\delta<\gamma < \alpha$, set
$U_{\alpha+1}:=U$, $f_{\alpha+1,\gamma}:= ff_{\gamma}$ for every $\gamma<\alpha$ and 
$f_{\alpha+1,\alpha} = f$, where $f \in \Hom(U_{\alpha},U_{\alpha+1})$ is injective and not 
surjective.
Assume that $\alpha$ is limit and the system $(U_{\gamma})_{\gamma < \alpha}$ has been defined. 
If the direct limit of this system is not isomorphic to $U$, set $\beta := \alpha$ and 
conclude. Otherwise, $(U_{\alpha},(f_{\alpha,\gamma})_{\gamma < \alpha})$ is defined by the 
direct limit of the system $(U_{\gamma})_{\gamma < \alpha}$.

Notice that each $U_\alpha$ contains a chain of proper submodules of length $\alpha$, so 
the construction must terminate.

Let $V$ be the direct limit of the constructed system. Then $V$ is a uniserial module which is not 
isomorphic to $U$, so $\Add(U)$ is not closed under direct limit. On the other hand, 
$V$ is a union of an infinite chain of submodules isomorphic to $U$.

Thus it is enough to show that $\Hom(U,V)$ does not contain a minimal weakly 
generating set. Assume that $X$ is a minimal weakly generating set. 
First observe that $X$ must contain a monomorphism: Let $f\in \Hom(U,V)$ 
be a monomorphism and assume $f = \sum_{x \in X} xs_x$ for some summable family indexed in $X$. Let $u\ne 0$ be an element of $U$. Then there exists a finite set  $X_u\subseteq X$
such that $s_x(u) = 0$ for any $x \in X\setminus X_u$. Then 
$f = \sum_{x \in X_u} xs_x+g$, where $g = \sum_{x \in X \setminus X_u} xs_x$ is not a monomorphism. 
Since $f$ is a monomorphism, there exists $x \in X_u$ which is a monomorphism. 

Moreover, $U$ and $V$ are not isomorphic and $\Hom(U,V)$ contains a monomorphism, so that 
there is no epimorphism in $\Hom(U,V)$. 
Observe that if $f,g \in \Hom(U,V)$ are monomorphisms, then $fS \subseteq gS$ or 
$gS \subseteq fS$. Indeed, let $U \stackrel{\pi_g}{\leftarrow} X\stackrel{\pi_f}{\rightarrow} U$
be a pullback of $U \stackrel{g}{\rightarrow} V \stackrel{f}{\leftarrow} U$. That is 
$f\pi_{f} = g \pi_g \colon X \to V$. Then $\pi_f,\pi_g$ are monomorphisms. If $f(U) \subseteq g(U)$, 
then $\pi_f$ is onto and $f = g \pi_g (\pi_f)^{-1} \in gS$. If $g(U) \subseteq f(U)$, then
$\pi_g$ is onto and $g = f\pi_f (\pi_g)^{-1} \in fS$.
On the other hand, $V$ is a union of submodules isomorphic 
to $U$, hence also $V = \bigcup_{x \in X} x(U)$. Then $X$ contains three elements 
$x_1,x_2,x_3$ such that $x_1$ is a monomorphism and $x_1(U) \subset x_2(U) \subset x_3(U)$.
For $i = 2,3$, let $y_i\in \{x_i,x_1+x_i\}$ be a monomorphism. Note  that 
$y_i(U) = x_i(U)$, and so $y_2 \in y_3S$. It follows that $x_2 \in x_1S + x_3S$, thus  $X$ is not minimal.   
\end{proof}

\section {The general case of artinian uniserial modules}\label{artinianuniserial}\label{4}

{\em In this section, $U$  will always be a non-zero artinian uniserial module. }
Therefore there exists an ordinal number 
$\alpha$ such that the chain of submodules of $U$ is 
$0 < U_1 < \dots < U_\alpha = U$. In this case, we will say that $U$ is of {\em length $ \alpha  + 1$.}  Clearly, $U$ is finitely generated if and only if 
$ \alpha$  is a successor, because in this case $U$ has a maximal submodule. 
The case of a finitely generated module $U$ was studied in Corollary \ref{f.g}. Therefore, we will now consider the case of $\alpha$ a limit ordinal. 

We will show that the behavior of the artinian uniserial module $U_R$ is mainly described, for our aims, by two ordinal numbers: the ordinal $\alpha$ (=the ordinal number such that $U_R$ is of length $\alpha+1$, that is, the lattice of submodules of $U_R$ is $0=U_0<U_1<\dots<U_\alpha=U_R$) and the ordinal $\beta$, defined by $\beta:=\min \{\,\gamma\mid \gamma\le\alpha,\ U_\gamma=\ker f$ for some $f\in S$, $\ker f\ne 0\,\}$. Thus we have a pair $(\alpha,\beta)$ of ordinal numbers attached to any artinian uniserial module $U_R$. We have that:

\medskip

(1) $\beta\le\alpha$.

\medskip

(2) $U_R$ is finitely generated if and only if $U_R$ is cyclic, if and only if the ordinal $\alpha$ is not a limit ordinal, if and only if cf$(\alpha)=1$, where cf$(\alpha)$ denotes the cofinality of $\alpha$.

\medskip

(3) $U_R$ is countably generated if and only if cf$(\alpha)\le\aleph_0$.  In particular, if cf$(\alpha)>\aleph_0$ or cf$(\alpha)=1$, then $U_R$ is small.

\medskip

(4) The following conditions are equivalent:
\begin{enumerate}
\item[(a)] $\beta \omega\ge\alpha.$
\item[(b)]  For any set $I$, the family $(U_i)_{i\in I}$,  where $U_i=U_R$ for every $i\in I$, is locally $T$-nilpotent.
\end{enumerate}
Moreover, if any of these two equivalent conditions holds, then $\Add(U)$ is a covering class.

\begin{proof} 
(a)${}\Rightarrow{}$(b). Suppose $\beta \omega\ge\alpha$. Let $x$ be a non-zero element of $U=U_\alpha$. Then $xR=U_\tau$ for some non-limit ordinal $\tau\le\alpha$ and $\beta\omega$ is a limit ordinal. Therefore $\tau<\beta\omega$. It follows that there is a finite ordinal $n$ such that $\tau\le\beta n$. Let $n$ be the smallest such finite ordinal, so that $\beta(n-1)<\tau\le\beta n$. Let us prove by induction on $t$ that $f_t\dots f_1(U_{\beta t})=0$ for every $t=1,2,\dots,n-1$ and every sequence $f_1,f_2,f_3,\dots$ of non-isomorphisms in $\End(U_R)$. We have that $f_1(U_\beta)=0$ because $U_\beta$ is contained in the kernel of all non-isomorphisms $U\to U$. Assume $f_t\dots f_{2}(U_{\beta (t-1)})=0$. It suffices to prove that $f_1(U_{\beta t})\subseteq U_{\beta(t-1)}$. Let $U_\delta$ be the kernel of $f_1$, so that $\delta\ge\beta$. By the correspondence theorem for submodules, there is a one-to-one correspondence $$\begin{array}{ccc}\Cal A:= \{\,U_\varepsilon\mid \delta\le\varepsilon\le\alpha\}& \leftrightarrow & \Cal B:=\{\, U_\vartheta\mid \vartheta\le\alpha,\ U_\vartheta\subseteq f_1(U)\,\}\\
X\in\Cal A&\mapsto & f_1(X)\in\Cal B\\
f_1^{-1}(Y) &\leftmapsto&Y\in\Cal B\\
U_{\delta+\vartheta} & \leftmapsto &U_\vartheta\in\Cal B.\end{array}$$ 
Assume $U_{\beta(t-1)} \subseteq f_1(U)$. Then $U_{\beta(t-1)}\in\Cal B$ corresponds to $U_{\delta+\beta(t-1)}\supseteq 
U_{\beta+\beta(t-1)}\supseteq U_{\beta t}
$. Thus $f_1(U_{\beta t})\subseteq U_{\beta(t-1)}$. This concludes the proof by induction on $t$. In particular, we get, for $t=n-1$, that $f_n\dots f_2(U_{\beta (n-1)})=0$. Similarly, the one-to-one correspondence above shows that  $f_1(U_\tau)\subseteq U_{\beta (n-1)}$. Thus $f_n\dots f_1(U_\tau)=0$, i.e.,  $f_n\dots f_1(x)=0$.

(b)${}\Rightarrow{}$(a).
Suppose $\beta \omega<\alpha.$ Let $f$ be an endomorphism of $U_R$ with $\ker(f)= U_\beta$. Consider the sequence $U\stackrel{f}{\rightarrow}U\stackrel{f}{\rightarrow}U\stackrel{f}{\rightarrow}\cdots$. Let $x$ be a generator of the cyclic module $U_{\beta\omega+1}\subseteq U_\alpha$. Let us show that $f^n(x)\ne 0$ for every $n\ge 1$. It suffices to show that $\ker(f^n)=U_{\beta n}$ for every $n\ge 1$. We prove this by induction on $n$. The case $n=1$ holds trivially. Suppose $U_{\beta(n-1)}=\ker(f^{n-1})$. By the correspondence theorem for submodules, there is a one-to-one correspondence $\{\,U_\varepsilon\mid \beta\le\varepsilon\le\alpha\} \leftrightarrow \{\, U_\vartheta\mid \vartheta\le\alpha,\ U_\vartheta\subseteq f(U)\,\}$. Note that if $f(U) \subseteq U_{\beta(n-1)}$, then 
$\alpha \leq \beta + \beta(n-1) = \beta n$, which is not possible. Hence  $U_{\beta(n-1)} \subseteq f(U)$,
so that $\ker(f^n)=f^{-1}(\ker(f^{n-1}))=f^{-1}(U_{\beta(n-1)})=U_{\beta+\beta(n-1)}=U_{\beta n}$.

Finally, if any of these two equivalent conditions holds, the class $\Add(U)$ is covering by Lemma~\ref{localy}.
\end{proof}

 As a consequence of (4), when $\alpha=\omega$, the class $\Add(U_R)$ is always covering.
 
\medskip

(5) $S$ is division ring if and only if $\alpha=\beta$.

\medskip

(6) Assume that the ordinal $\alpha$ is additively indecomposable, i.e., of the form $\omega^\gamma$ for some ordinal $\gamma$. Then every non-zero endomorphism of $U_R$ is surjective and $S$ is a left chain domain.

\begin{proof}
An ordinal number 
$ \alpha$  is an additively indecomposable ordinal, that is, 
$\beta,\beta' < \alpha$ implies  $\beta+ \beta' < \alpha$, if and only if $\alpha=\omega^\gamma$ for some ordinal $\gamma$. 
It follows that every non-zero endomorphism of $U$
 is surjective (because (1) non-zero endomorphism means with a kernel $U_\beta$, where  
  $\beta < \alpha$, and (2)
 the first isomorphism theorem says that $\beta +  \beta' = \alpha,$ where $U_
 {\beta'}$  is the image of the endomorphism, so we must have $\beta'=\alpha,$
 i.e., the morphism is surjective.) As the composite mapping of two surjective mappings is a surjective mapping, we get that
 $S$ is a domain. 
We will now show that $S$ is a left chain ring. 
Take any two non-zero endomorphisms $\varphi , \psi $ of  $U.$ Then 
$\varphi , \psi $ are both surjective, and without loss of generality we can suppose that $\ker \varphi\subseteq \ker\psi$. Since $\varphi$ is onto there exists $\tau \in S$ such that 
$\psi = \tau \varphi$, i.e.,\linebreak $\psi \in S \varphi$.\end{proof}

(7) The Jacobson radical $J(S)$ is nilpotent if and only if $\beta n\ge\alpha$ for some finite ordinal $n$.

\bigskip

As we have just seen, the pair of ordinal numbers $(\alpha,\beta)$ allows us to describe some properties of the artinian uniserial module $U_R$ and its endomorphism ring $S$. We now refine this idea.

It is easily seen that, for every fixed ordinal $\alpha$, the set $\alpha+1$ of all ordinals $\le\alpha$ is a monoid with respect to the operation $\oplus\colon (\alpha+1)\times(\alpha+1)\to \alpha+1$ defined, for every $\gamma,\delta\le\alpha$, by $\gamma\oplus\delta:=\min\{\gamma+\delta,\alpha\}$. The ordinal $0$ is the identity of the monoid. The ordinal $\alpha$ is a zero element of the monoid. The monoid $(\alpha+1,\oplus)$ is a left cancellative monoid with zero, in the sense that $\gamma\oplus\delta=\gamma\oplus\delta'\ne\alpha$ implies $\delta=\delta'$. There is a monoid antihomomorphism $\kappa\colon S\to\alpha+1$ that associates to every $f\in S$ the ordinal $\gamma\le\alpha$ such that $\ker(f)=U_\gamma$. (Recall that the submodules of $U_R$ are $0=U_0<U_1<\dots<U_\alpha$.) We extend the definition of $\kappa$ to all homomorphisms with domain $U$.
If $f \in \Hom(U,V)$, where $V$ is an arbitrary $R$-module, we define 
$\kappa(f) = \beta$ if $\ker(f) = U_{\beta}$.  
It is now easy to see that $\kappa$ is a monoid antihomomorphism (Lemma~\ref{3.1}).

\begin{Lemma}\label{3.1} Let $f\colon U_R\to V_R$ and $f'\colon V_R\to W_R$ be morphisms between artinian uniserial right $R$-modules $U_R,V_R,W_R$. If $\ker(f)=U_\beta$ and $\ker(f')=V_{\beta'}$, then $\ker(f'\circ f)=U_{\beta\oplus \beta'}$\end{Lemma}

\begin{proof} Trivially, $\ker(f'\circ f)=f{}^{-1}(\ker(f'))=f{}^{-1}(V_{\beta'})$.
By the correspondence theorem for submodules, the morphism $f$ induces a one-to-one correspondence $$\begin{array}{ccc}\Cal A:= \{\,U_\varepsilon\mid \beta\le\varepsilon\le\alpha\}& \leftrightarrow & \Cal B:=\{\, V_\vartheta\mid \vartheta\le\alpha',\ V_\vartheta\subseteq f(U)\,\}\\
X\in\Cal A&\mapsto & f(X)\in\Cal B\\
f^{-1}(Y) &\leftmapsto&Y\in\Cal B\\
U_{\beta\oplus\vartheta} & \leftmapsto &V_\vartheta\in\Cal B.\end{array}$$ 
If $f(U)\subseteq \ker(f')$, then $\alpha \leq \beta + \beta'$. Therefore  
$\ker (f'\circ f) = U_{\alpha} = U_{\beta \oplus \beta'}$. If $\ker(f') \subseteq f(U)$, 
the inverse image via $f$ of $\ker(f')=V_{\beta'}$ is $U_{\beta\oplus \beta'}$.\end{proof}

As $\kappa\colon S\to\alpha+1$ is a monoid antihomomorphism, its image $\kappa(S)$ is a submonoid with zero of the monoid $(\alpha+1,\oplus)$. The cardinal $\beta$ considered in the first part of this section is exactly $\beta=\min(\kappa(S)\setminus\{0\})$. It is in this sense that considering the submonoid $\kappa(S)$ of $(\alpha+1,\oplus)$ is a refinement of the idea of considering the pair $(\alpha,\beta)$.

Notice that the antihomomorphism $\kappa\colon S\to\alpha+1$ determines not only the kernel of any $f\in S$ ($\ker(f)=U_{\kappa(f)}$), but also the image of $f$, because $f(U_R)=U_\gamma$, where $\gamma$ is the unique ordinal such that $\kappa(f)+\gamma=\alpha$. Moreover, $\kappa(f)$ determines suitable factorizations of $f$, as the next lemma shows.

\begin{Lemma} Let $f\colon U_R\to V_R$ and $f'\colon U_R\to W_R$ be morphisms between artinian uniserial right $R$-modules $U_R,V_R,W_R$. Then $\kappa(f)\le\kappa(f')$ if and only if $f'=g\circ f|^{V_\gamma}$ for some morphism $g\colon V_\gamma\to W$. Here $f|^{V_\gamma}\colon U\to V_\gamma$ denotes the corestriction of $f$ to the image $V_\gamma$ of $f$.\end{Lemma}




\begin{Lemma} Let $f\colon U_R\to V_R$ be an $\Add(U_R)$-precover, where $U_R$ is artinian uniserial. Then $\kappa(f)=\min\{\,\kappa(g)\mid g\in\Hom(U_R,V_R)\,\}$.\end{Lemma}

\begin{proof} If $f\colon U_R\to V_R$ is an $\Add(U_R)$-precover, then every morphism\linebreak $g\in\Hom(U_R,V_R)$ factors through $f$, i.e., $g=fs$ for some $s\in S$. By Lemma \ref{3.1} $\kappa(g) = 
\kappa(s)\oplus\kappa(f)$, so $\kappa(f) \leq \kappa(g)$.
\end{proof}

\begin{proposition}\label{3.5} Let $\kappa(S)$ be the submonoid of $\alpha+1$ associated to the endomorphism ring $S$, so that $\kappa(S)\cong\vartheta+1$ as an ordered set for some ordinal $\theta\le\alpha$. Then $U$ is not self-small if and only if $\kappa(S\setminus\{0\})$ is cofinal in $\alpha$ and cf$(\vartheta)=\aleph_0$.\end{proposition}

\begin{proof} Suppose $\kappa(S\setminus\{0\})$ cofinal in $\alpha$ and cf$(\vartheta)=\aleph_0$. Then $\bigcup_{0 \neq f\in S}\ker(f)=U_R$ and there exists a countable family of elements $f_i,\ i\in\N$, of $S\setminus\{0\}$ such that for every $f\in J(S)\setminus\{0\}$ there exists $i\in\N$ with $\ker(f)\subseteq\ker(f_i)$. Hence $\bigcup_{i\in \N}\ker(f_i)=U_R$. Therefore it is possible to extract from the sequence $f_i$ a subsequence $f_{i_n},\ n\in\N,$ with $\ker(f_{i_1})\subset \ker(f_{i_2})\subset \ker(f_{i_3})\subset \cdots.$ Then the morphism $(f_{i_n})_{n\in\N}\colon U\to U^{(\N)}$ shows that $U$ is not self-small.

Conversely, suppose that $U$ is not self-small. Then there exists a morphism $U\to U^{(I)}$ that shows that $U$ is not self-small, $I$ an infinite set. Without loss of generality, $I=\N$. Let $f=(f_n)_{n\in\N}\colon U\to U^{(\N)}$ be such a morphism, so that, for every $x\in U_R$, $f_n(x)=0$ for almost all $n\in\N$, and $f_n\ne 0$ for all $n\in\N$. Then $\{\,\ker(f_n)\mid n\in\N\,\}$ is cofinal in $U_R=\bigcup_{n\in\N}\ker(f_n)$. Therefore cf$(\vartheta)\le\aleph_0$ and $\kappa(S\setminus\{0\})$ is cofinal in $\alpha$. Suppose cf$(\vartheta)<\aleph_0$. Then cf$(\vartheta)=1$, i.e., $\vartheta $ is a non-limit ordinal, that  is, $\kappa(S\setminus\{0\})$ has a greatest element. Hence there exists $f\in S$, $f\ne0$, such that $\ker(f)\supseteq\ker(g)$ for every $g\in S$, $g\ne0$. If $x\in U\setminus\ker(f)$, then $U_R\supset\ker(f)\supseteq \bigcup_{n\in\N}\ker(f_n)$, a contradiction. This proves that cf$(\vartheta)=\aleph_0$.
\end{proof}

\begin{theorem}
Let $U$  be a uniserial artinian right $R$-module of length  
$\omega+1.$ Then $U$  is self-small if and only if $ S = \End(U)$  is a division ring.
\end{theorem}

\begin{proof} Let $U$  be a uniserial artinian self-small right $R$-module of length 
$\omega+1$. We can apply Proposition~\ref{3.5} to the module $U_R$. In this specific case, we have that $\alpha=\omega$, so that $\kappa(S\setminus\{0\})$ is a submonoid of the additive monoid $\N$. All non-zero additive submonoids of $\N$ are infinite. Thus we have that either $\kappa(S\setminus\{0\})=0$ or $\kappa(S\setminus\{0\})$ is cofinal in $\N$. If $\kappa(S\setminus\{0\})=0$, then all non-zero endomorphisms in $S$ are monic, hence isomorphisms, so $S$ is a division ring in this case.
If $\kappa(S\setminus\{0\})$ is cofinal in $\N$, then $\vartheta=\omega$ in the notation of Proposition~\ref{3.5}, hence cf$(\vartheta)=\aleph_0$, so $U$ is not self-small by Proposition~\ref{3.5}, a contradiction. This proves that $ S = \End(U)$  is a division ring.

Conversely, suppose $S$ a division ring. Fix any morphism $(f_i)_{i\in I}\colon U\to U^{(I)}$. For any non-zero $x\in U$, we have that $f_i(x)=0$ for almost all $i$. But $x\ne0$ and $f_i(x)=0$ implies $f_i=0$, because all non-zero elements of $S$ are automorphisms of $U$. Hence $f_i=0$ for almost all $i\in I$. This proves that $U$ is self-small. \end{proof}

\section{The case $ K\subset I $}\label{5}

Throughout this section, {\em 
$U_R$ is a uniserial $R$-module and $S := {\rm End}(U_R)$. 
We assume that $K\subset I \subset S$, i.e., every monomorphism in $S$ is onto and 
there exists an epimorphism in $S$ which is not monic.  
Moreover, we assume that the set $\{\,\ker f \mid f \in J(S)\, \}$ has a least element. 
This least kernel will be denoted by $V$. Observe that if 
$f \in S$, then $f(V)=0$ is equivalent to $f \in J(S)$.}
Note that every uniserial artinian module $U$ with $U_e \neq 0$ satisfies our assumptions.

\begin{lemma} \label{Vproperties} In the notation introduced above, the following statements hold.

\begin{enumerate}
  \item[{\rm (1)}] There exists an epimorphism $\varphi \in S$ such that 
$V = \ker \varphi$. In particular, $J(S) = S \varphi$.
  \item[{\rm (2)}] $V$ is a fully invariant submodule of $U$
  \item[{\rm (3)}] Let $\varphi \in S$ be an epimorphism with kernel $V$ and $V_i := \ker \varphi^i$.
            Then $V_i$ is a fully invariant submodule of $U$ for every $i \in \mathbb{N}$.
	\item[{\rm (4)}] If $f \in S$ satisfies $0 \neq \ker f \varsubsetneq V_{\omega}:= \sum_{i \in \N} 
	          V_i$, then $\ker f = V_i$ for some $i \in \mathbb{N}$.
	\item[{\rm (5)}]The family $(U_i)_{i \geq 1}$, where $U_i = U, i \in \mathbb{N}$,  
            is locally $T$-nilpotent if and only $V_{\omega} = U$.
\end{enumerate}
\end{lemma}

\begin{proof}
(1) Clearly, there exist $f \in S$ such that $\ker f = V$ and 
 an epimorphism $g \in J(S)$. Observe that $\ker g$
contains $V$. If $f$ is onto, 
then $\varphi$  can be $f$. If $\ker g = V$, $\varphi$ can be $g$. 
If $f$ is not onto and $\ker g \neq V$, set $\varphi := f + g$.
The last statement in (1) follows from the homomorphism theorem. 
\smallskip

(2) Let $f\in S$ be a morphism with kernel $V$. Assume there exist $g \in S$ and 
$v \in V$ such that $g(v) \not \in V$. Then $fg \in J(S)$ but $fg(v) \neq 0$. So 
$fg$ is an element of $J(S)$ having its kernel strictly contained in $V$. This is not 
possible. 
\smallskip

(3) Let $g \in S$ be such that $g(v) \not \in V_i$ for some $v \in V_i$. That is, 
$\varphi^ig(v) \neq 0$. Note that $\varphi g \in J(S)$ and, by (1), $J(S) = S \varphi$.
Applying this rule $i$ times, we get that $\varphi^i g = g'\varphi^i$ for some $g'\in S$.
Evaluating both sides on $v$ we get a contradiction $0 \neq \varphi^i g (v) = g' \varphi^i(v) = 0$.
\smallskip

(4) As before, let $f$ be a non-zero element of $J(S)$ and  $\varphi\in S$ be an 
epimorphism with kernel $V$. By (1), there exists $f_1 \in S$ such that $f = f_1 \varphi$.
If $f_1$ is not monic, we can write $f_1 = f_2\varphi$ for some $f_2 \in S$. Hence 
$f = f_2\varphi^2 $. Repeating this process, we find $f_1,f_2,\dots$ such that 
$f = f_k \varphi^k$ for each $k$. Since the kernel of $f$ is strictly contained in 
$V_{\omega}$, this process has to stop. That is, for some $t \in \mathbb{N}$ we have 
$f = f_t \varphi^t$, where $f_t$ is a monomorphism. Then, obviously, $\ker f = 
V_t$.
\smallskip

(5) If $V_{\omega} \neq U$, consider the sequence 
$$U \stackrel{\varphi}{\to} U \stackrel{\varphi}{\to} U \stackrel{\varphi}{\to} \cdots,$$
where $\varphi$ is an epimorphism with kernel $V$. If $V_{\omega} \neq U$, this sequence shows that the family considered is not locally $T$-nilpotent.

Conversely, fix a sequence of non-isomorphisms 
$$U\stackrel{f_1}{\to} U\stackrel{f_2}{\to} U\stackrel{f_3}{\to} \cdots $$ 

By (1), we can write $f_i = g_i \varphi$ for some $g_i \in S$. Then 
for every $n \in \mathbb{N}$ there exists $h_n$ such that 
$f_n \cdots f_2f_1 = h_n \varphi^n$. So $f_n\cdots f_1 (V_n) = 0$ for every 
$n$. Since $U = \bigcup_{n  \geq 1} V_n$, the family of modules considered  is 
locally $T$-nilpotent. 

\end{proof}

\begin{remark} {\rm 
Notice that the modules $V_i,\ i \in \N,$ and $V_{\omega}$ were defined via the 
endomorphism $\varphi \in S$. But there exists another description of these submodules 
which is independent of $\varphi$. Namely, $V_i = \{\,u \in U \mid  f(u) = 0$  for every $ f \in J^{i}\,\}$. 

The module $V_{\omega}$ is a fully invariant submodule of $U$ contained in $U_e$. In general, 
$V_{\omega} \neq U_e$. As an example, consider the artinian module of length $\omega^2+1$ having 
all its non-zero factors isomorphic presented in \cite {Facsal}. }
\end{remark}

\smallskip

\begin{lemma} \label{independence}
Assume that $A$ is an $R$-module such that the $S$-module $\Hom(U_R,A_R)$ 
has a minimal weakly generating set $X$.  
Then $X$ satisfies the following independence property: If $\{\,s_x\mid x \in X\,\}$ is a
summable family of $S$ such that $\sum_{x \in X} xs_x = 0$. Then 
$s_x(V_{\omega}) = 0$ for every $x \in X$.
\end{lemma}

\begin{proof} Let $\varphi \in S$ be an epimorphism with kernel $V$.
Assume that $\sum_{x \in X} xs_x = 0$ and $s_x(V_{\omega}) \neq 0$ for some 
$x \in X$. Let $x'\in X$ be such that $\ker s_{x'} = {\rm min} \{\ker s_x \mid x \in X\}$. 
Note that $0 \neq \ker s_{x'}$ (since $X$ is minimal) and $\ker s_{x'} = V_i$ 
for some $i \in \mathbb{N}$ by Lemma \ref{Vproperties}(4). Then for every $x \in X$ there exists $t_x \in S$ such that 
$s_x = t_x\varphi^i$. The family $\{\,t_x\mid x \in X\,\}$ is also summable and 
$0 = (\sum_{x \in X} xt_x)\varphi^{i}$. But $\varphi$ is onto, so
$\sum_{x \in X} xt_x =0$. As $t_{x'}$ is a monomorphism, we get a contradiction
to the minimality of $X$.
\end{proof}

Now, in the notation of Lemma~\ref{independence}, we want to define a map $\pi_x \colon \Hom(U_R,A_R)\to T:=\Hom_R(V_{\omega},V_{\omega})$ for each $x \in X$. Note that $T$ has a canonical right $S$-module structure, because $V_{\omega}$ is a fully invariant 
submodule of $U$: for every $t \in T$ and $s \in S$ define $ts:= ts|_{V_{\omega}}$.
 
If $f \in \Hom(U_R,A_R)$, express $f$ as $f = \sum_{x \in X} xs_x$, where $\{\,s_x\mid x \in X\,\}$ 
is a summable family of $S$. Now set $\pi_x(f) := s_x|_{V_{\omega}}$. 
By Lemma \ref{independence}, $\pi_x$ is well defined, and  it is easy to verify that $\pi_x$ is a 
homomorphism of $S$-modules. Note that, for every $f \in \Hom_R(U,A)$, the family 
$\{\,\pi_x(f)\mid x \in X\,\}$ is summable.  If $V_{\omega} \neq U$, almost all the morphisms in 
this family are zero, because $s_x|_{V_{\omega}} = 0$ whenever $s_x(u) = 0$ for 
some $u \in U \setminus V_{\omega}$. 

\medskip

The following result is similar to results by S.~U.~Chase in \cite{Chase}. We keep the notations of all this section. In particular, $T:=\Hom_R(V_{\omega},V_{\omega})$.

\begin{theorem} \label{products} 
Set $X_0 := \{\,x \in X \mid \pi_x \neq 0\,\}$. Assume $U \neq V_{\omega}$. 
Then $\Hom_R(U,A)$ is not isomorphic to any module of the form $\prod_{i \in I} N_i$, where 
\begin{enumerate}
\item[(a)] $I$ is of infinite  cardinality $|I|\ge |T|$.
\item[(b)] $N_j\varphi^n \neq N_j\varphi^{n+1}$ for every $n \in \mathbb{N}_0$, every $j \in I$ and every epimorphism 
$\varphi \in S$ with kernel $V$.
\end{enumerate}
\end{theorem}

\begin{proof}
Assume $\Hom_R(U,A) = M = \prod_{i \in I} N_i$.
We start with a sequence $I = I_0 \supset I_1 \supset I_2 \supset \dots$
such that all  the subsets $I_n$ have cardinality $|I|$ and $\bigcap_{n \in \mathbb{N}_0} 
I_n = \emptyset$. Set $M_n:= \prod_{j \in I_n} N_j$. We will view $M_n$ canonically 
as a submodule  of $M$.

Now consider the following construction.
Start with an element $m_0 \in M \setminus M \varphi$. Note that there 
exists an index $i_0 \in X_0$ such that $\pi_{i_0} (m_0)$ is a monomorphism. 
On the other hand, there exists a finite set $F_0 \subseteq X_0$ such that 
$\pi_x (m_0) = 0$  every $x \in X_{0} \setminus F_0$.

Using the cardinality argument by Chase (namely, there are two elements of $M_1\varphi$ 
having different classes in the group $M_1\varphi /M_1 \varphi^2$ and the same image 
under $\pi_x$ for every $x \in F_0$, because $F_0$ is finite and $M_1\varphi/M_1\varphi^2$ has 
at least $2^{|T|}$ elements), we get that
 there exists $m_1 \in M_1 \varphi \setminus M_1\varphi^2$
such that $\pi_x(m_1) = 0$ for every $x \in F_0$. Since $m_1 \in M\varphi$ and
$m_1 \not \in M \varphi^2$, there exists $i_1 \in X_0$ such that $\pi_{i_1}(m_1)$ has 
kernel $V_1$. Of course $i_1 \not \in F_0$, in particular $i_1 \neq i_0$. 
Also, there exists a finite set $F_1 \subseteq X_0$ such that $\pi_x(m_1) = 0$ 
for every $x \in X \setminus F_1$.

It is now clear how to continue. At the $n$-th step, define $m_n \in M_n\varphi^n \setminus 
M_n \varphi^{n+1}$ such that $\pi_x(m_n) = 0$ for every $x \in F_0 \cup F_1 \cup \dots \cup F_{n-1}.$
Then there exists $i_n \in X_0$ such that $\pi_{i_n}(m_n)$ has kernel $V_n$. Note that  $i_n\not \in 
\{i_0,\dots,i_{n-1}\}$. Moreover, there exists a finite set $F_{n} \subseteq X_0$ such 
that $\pi_x(m_n) = 0$ for every $x \in X_0 \setminus F_n$. 

In the end we get a sequence of elements $m_0,m_1,m_2,\dots$, where $m_n \in M_n$. In $M = \prod_{i\in I} N_i$
we can define $m := \sum_{n\in \mathbb{N}} m_n$ because $\bigcap_{n \in \mathbb{N}_0} I_n = \emptyset $. We claim that, for every $n \in \mathbb{N}$,
$\pi_{i_n}(m)$ has kernel $V_n$. Consider $u_n := \sum_{i>n} m_i$ and note that 
$u_n \in M \varphi^{n+1}$. Then $m = m_1+m_2+\cdots+ m_{n-1} + m_n + u_n$. Since 
$i_n \not \in F_0\cup \cdots \cup F_{n-1}$, we have $\pi_{i_n}(m_1+ \cdots +m_{n-1}) = 0$.
By our construction, $\pi_{i_n}(m_n)$ has kernel $V_n$ and $\pi_{i_n}(u_n)$ is in 
$T\varphi^{n+1}$, so its kernel strictly contains $V_{n}$. Hence the morphism $\pi_{i_n}(m)$ has 
kernel $V_n$.

Now it is easy to conclude. Write $m = \sum_{x \in X} xs_x$, where 
$\{\,s_x\mid x \in X\,\}$ is a summable family in $S$. Notice that $s_x|_{V_\omega} = \pi_x(m)$.
We know that $s_{i_n}$ has kernel $V_{n}$. So, if $u \in U \setminus V_{\omega}$, we get that 
$s_{i_n}(u) \neq 0$ for every $n \in \mathbb{N}$. This contradicts the 
summability.
\end{proof}

\begin{theorem}\label{lastcases}
Under the hypotheses in this section on $U$, one has that  $\Add(U)$ is covering if and only if it is closed under direct limit. 
\end{theorem}
\begin{proof}
Let $\Add(U)$ be covering. By Lemma \ref{localy}, it is enough to show that 
the family $(U_i)_{i \geq 1}$, where $U_i = U$ for every $ i \in \mathbb{N}$,  
            is locally $T$-nilpotent. 
Assume the contrary. By Lemma~\ref{Vproperties}(5), $U \neq V_{\omega}$. 
In Theorem \ref{products}, apply  $A = U^I$, where 
$I$ is sufficiently large. Then $\Hom_R(U,U^{I}) \simeq S^{I}$, and 
$S \varphi^{n} \neq S \varphi^{n+1}$ for every $n \in \mathbb{N}$. 
Therefore $U^{I}$ has not an $\Add(U)$-cover.   
\end{proof}

Let  $U$  be an artinian uniserial module. Let $\alpha+1$ be the length of $U$ ($\alpha$ an ordinal number), so 
that the chain of submodules of $U$ is 
$0 < U_1 < \dots < U_\alpha = U$.  
If $U\ne0$ is a uniserial artinian module of length $\alpha+1$, and 
 $\beta:=\min \{\,\gamma\mid \gamma\le\alpha,\ U_\gamma=\ker f, \ker f\ne 0, f\in S\}$,
 then $U$ is not isomorphic to any 
proper factor if and only if $\beta+\alpha>\alpha$.
In fact, if $U$ is isomorphic to a proper factor $U/V$,
then there is an endomorphism $f$ of $U$, surjective and with kernel $V$. Hence,
if $V= U_\gamma$, then $\gamma+\alpha=\alpha$. But $\beta\le\gamma,$ so $\beta+\alpha\le \alpha$.
Conversely, if we don't have $\beta+\alpha > \alpha$, then clearly $\beta+\alpha = \alpha$. Hence if
$g$ is the endomorphism of $U$ with kernel $U_\beta,$ then $g$ is surjective, hence $U/U_\beta\cong U$.

We can finally state the third of the main results of this paper:

\begin{theorem}\label{5.6}
If $U$ is a uniserial artinian module, then $\Add(U)$ is covering if and only if it is closed under direct limit. 
\end{theorem}
\begin{proof}
Let $U$  be an  artinian uniserial module of length $ \alpha  + 1$. As we have said above, we can assume that $\alpha$ is limit (if $\alpha$ is successor, then 
 $U$ is finitely generated and we are done). By Theorem~\ref{Ue},
$\Add(U)$  covering implies that either $U_e = 0$ or
$U_e = U$. We claim that $U_e \neq 0$. 
Otherwise,  by the above remark, $\beta+\alpha>\alpha$. As $\alpha$ is limit, this is impossible. 
Therefore  $U_e \neq 0$ and we can apply  Theorem~\ref{lastcases}. 
\end{proof} 

\section{Examples}\label{6}

In view of Theorem~\ref{5.6}, in this section we give some examples of artinian uniserial modules describing their endomorphism ring. 
For every ordinal $\alpha$, there exist 
 artinian uniserial modules of length $ \alpha + 1$ \cite[Example~3]{Fuchs}.
 
Now let $R$ and $T$ be any two rings and $_TU_R$ be a $T$-$R$-bimodule. Let $\Cal L(_TU)$ be the lattice of all submodules of the left $T$-module $_TU$ and $\Cal L(R_R)$ be the lattice of all right ideals of $R$. There is a mapping $\alpha\colon \Cal L(_TU)\to \Cal L(R_R)$, defined, for every $X\in  \Cal L(_TU)$, by $\alpha(X)=r.ann_R(X)=\{\,r\in R\mid xr=0$ for every $x\in X\,\}$. This mapping $\alpha$ is inclusion reversing, that is, 
$X\subseteq Y$ implies $\alpha(X) \supseteq \alpha(Y)$.

Let $\Cal L_c(_TU)$ be the partially ordered subset of $\Cal L(_TU)$ consisting of all cyclic submodules $Tu$ of the left $T$-module $_TU$, and let $\alpha|_c\colon \Cal L_c(_TU)\to \Cal L(R_R)$ be the restriction of $\alpha$ to $\Cal L_c(_TU)$.

\begin{Lemma} If the module $U_R$ is injective and $T:=\End(U_R)$, then the mapping $\alpha|_c\colon \Cal L_c(_TU)\to \Cal L(R_R)$ is injective. Moreover,  for every $X,Y\in  \Cal L_c(_TU)$, $X\subseteq Y$ if and only if $\alpha(X) \supseteq \alpha(Y)$.\end{Lemma}

\begin{proof} Let us first prove that, for every $X,Y\in  \Cal L_c(_TU)$, $\alpha(X) \supseteq \alpha(Y)$ implies $X\subseteq Y$. If $X,Y\in  \Cal L_c(_TU)$, then $X=Tu$ and $Y=Tv$ for suitable $u,v\in{}_TU_R$. Now $\alpha(X) \supseteq \alpha(Y)$, i.e., $r.ann_R(X)=r.ann_R(Tu)=r.ann_R(u)\supseteq r.ann_R(Y) = 
r.ann_R(Tv)=r.ann_R(v)$. Since
$r.ann (u) \supseteq r.ann (v)$, there is the canonical projection $R_R/r.ann (v) \to R_R/ r.ann (u)$, 
which is a right $R$-module morphism. Hence there is a right $R$-module morphism $vR\to uR$, $v\mapsto u$, which 
extends to the injective right $R$-module $U_R   \to U_R $, so that there is a right $R$-module 
morphism $t\colon U_R\to U_R$ with $t(v)=u$, i.e., an element $t\in T$ with $tv=u$. Therefore 
$X=Tu=Ttv \subseteq Tv=Y$. 

As far as the injectivity of $\alpha|_c$ is concerned, suppose $X,Y\in  \Cal L_c(_TU)$ and $\alpha(X) = \alpha(Y)$. Then $\alpha(X) 
\supseteq \alpha(Y)$ and $\alpha(X) \subseteq \alpha(Y)$, so that $X\subseteq Y$ and $X\supseteq Y$. Therefore 
$X=Y$ and $\alpha|_c$ is an injective mapping.\end{proof}

{\em From the next proposition to the end of Remark~\ref{rem}, we will consider the case of a right noetherian right chain ring $R$ and $U_R$ an injective $R$-module. In the next proposition, we will prove that $U$ turns out to be an artinian uniserial left module over its endomorphism ring $T:=\End(U_R)$. Since $R$ is a right noetherian right chain ring, its right ideals form a well ordered descending chain $R=I_0\supset I_1\supset I_2 \supset \dots\supset I_{\alpha}=0$. }This notation will be kept until the end of Remark~\ref{rem}.

\begin{proposition}\label{6.2} Let $R$ be a right noetherian right chain ring, $U_R$ be an injective module and $T:=\End(U_R)$. Then $_TU$ is an artinian uniserial left $T$-module.\end{proposition}

\begin{proof} By the previous Lemma, $\Cal L_c(_TU)$ is order antiisomorphic to a partially ordered subset of $\Cal L(R_R)$. Now $\Cal L(R_R)$ is a linearly ordered noetherian set, so that $\Cal L_c(_TU)$ is a linearly ordered artinian set. From this, it follows that $_TU$ is  uniserial artinian \cite [Theorem 10.20] {libro}. 
\end{proof}

In the proof of the previous Proposition, notice that if $\Cal L(_TU)$ is isomorphic to an ordinal $\alpha$, then $\Cal L_c(_TU)$ consists of all ordinals $\beta<\alpha$ that are not limit ordinals or are $0$. 

\begin{proposition} Let $R$ be a right noetherian right chain ring, $U_R$ be the injective module of the unique simple right $R$-module and $T:=\End(U_R)$. Then the artinian linearly ordered set $\Cal L(_TU)$ is antiisomorphic to the noetherian linearly ordered set $\Cal L(R_R)$.\end{proposition}

\begin{proof} Consider the order reversing mapping $\alpha\colon \Cal L(_TU)\to \Cal L(R_R)$ and its restriction $\alpha|_c\colon \Cal L_c(_TU)\to \Cal L(R_R)$. The image of $\alpha|_c$ consists of all right ideals $I\in \Cal L(R_R)$ such that $R_R/I$ has a non-zero socle (equivalently, an essential socle, because $R_R/I$ is uniserial). These are the right ideals $I\in \Cal L(R_R)$ that have an immediate successor in $\Cal L(R_R)$. If $\Cal L(R_R)$ is order antiisomorphic to an ordinal $\alpha$, then $\alpha|_c\Cal L_c(_TU)$ corresponds to the set of all ordinals $\beta<\alpha$ that are not limit ordinals or are $0$. Hence $\alpha\colon \Cal L(_TU)\to \Cal L(R_R)^{op}$ is an increasing mapping between two well-ordered sets which induces an antiisomorphism between the elements with an immediate predecessor. Thus $\alpha\colon \Cal L(_TU)\to \Cal L(R_R)^{op}$ is an isomorphism of  well-ordered sets.\end{proof} 

\begin{proposition} Let $R$ be a right noetherian right chain ring with a non-zero
 right socle, that is, let $R$ be a ring with $\Cal L(R_R)$ antiisomorphic to $\alpha+1$ for some non-limit ordinal $\alpha$. Let $U_R$ 
be the injective envelope of the unique simple right $R$-module, $T:=\End(U_R)$ and $S:=\End(_TU)$. Then $S$ is 
canonically isomorphic to $R$, that is, every left $T$-module endomorphism of $_TU$ is given by right multiplication 
by a unique element of $R$.
\end{proposition}

\begin{proof} 
Since $R_R$ has an essential simple socle, we have that $U_R=E(R_R)$, 
the injective envelope of $R_R$. As a consequence, $_TU$ is a cyclic left $T$-module, generated by $1_R$. Hence every left $T
$-module endomorphism $f$ of $_TU$ is completely determined by the image $(1_R)f$ of $1_R$, which is an 
element of $U$. Assume by contradiction that $ (1_R)f\notin R_R$. Since $U_R$ is the minimal injective cogenerator 
in $\Mod R$, there exists a morphism $t'\colon U_R/R_R\to U_R$ that maps the non-zero element $(1_R)f+ R_R
$ of $U_R/R_R$ to a non-zero element of $U_R$. Hence there exists an endomorphism $t\colon U_R\to U_R$ such 
that $t(1_R)=0$ and $t((1_R)f)\ne 0$. Thus $0\ne t(1_R)f=(t(1_R))f= (0)f=0$, which is a contradiction. The 
contradiction shows that $(1_R)f\in R_R$. Set $r:=(1_R)f$. Then $(x)f=xr$ for every $x\in {}_TU$, because 
every element $x\in {}_TU$ is of the form $t'1_R$ for some $t'\in T$. The element $r\in R$ is unique, because the 
minimal injective cogenerator $U_R$ is faithful.
\end{proof}

\begin{proposition} Let $R$ be a right noetherian right chain ring with a zero right socle, that is, let $R$ be a ring with $\Cal L(R_R)$ antiisomorphic to $\alpha+1$ for some limit ordinal $\alpha$ via the antiisomorphism $\gamma<\alpha\mapsto I_\gamma$. Let $U_R$ be the injective envelope of the unique simple right $R$-module, $T:=\End(U_R)$ and $S:=\End(_TU)$. Then $S$ is canonically isomorphic to $\displaystyle \lim_{\longleftarrow} {}R/I_{\gamma+1}$, the completion of $R$ in the topology on $R$ with basis of neighborwoods of $0$ the ideals $I_{\gamma+1}$, where $\gamma+1$ ranges in the set of non-limit ordinals $\gamma+1\le\alpha$.\end{proposition}

\begin{proof} We know that $_TU$ is a uniserial artinian $T$-module (Proposition~\ref{6.2}). Its $T$-submodules are in one-to-one 
correspondence with the right ideals of $R$. In correspondence to the ideals $I_{\gamma+1}$, there are the cyclic $T
$-submodules $l.ann_U(I_{\gamma+1})$ of $_TU$, and $_TU=\bigcup_{\gamma+1}l.ann_U(I_{\gamma+1})$. Hence 
$$S=\End_T(\bigcup_{\gamma+1}l.ann_U(I_{\gamma+1}),{}_TU) \cong 
 \displaystyle \lim_{\longleftarrow} {}\Hom(_Tl.ann_U(I_{\gamma+1}),{}
_TU). $$ 
Any $T$-module morphism $_T l.ann_U(I_{\gamma+1})\to {}_TU$ 
induces an order preserving mapping of 
$\Cal 
L(ann_U(I_{\gamma+1}))\cong\gamma+1\to\Cal L(_TU)\cong\alpha+1$. It follows that any $T$-module morphism 
$_Tl.ann_U(I_{\gamma+1})\to {}_TU$ has its image contained in $_Tl.ann_U(I_{\gamma+1})$, that is, every 
$_Tl.ann_U(I_{\gamma+1})$ is a fully invariant left $T$-submodule of $_TU$. This proves that every 
$_Tl.ann_U(I_{\gamma+1})$ 
is a right submodule of $U_S$. 
Thus $S\cong \displaystyle \lim_{\longleftarrow} {}\End(_Tl.ann_U(I_{\gamma+1})).$
 Now $l.ann_U(I_{\gamma+1})$ is the minimal injective 
cogenerator for the category $\Mod R/I_{\gamma+1}$, and $R/I_{\gamma+1}$ is a right noetherian right chain ring 
with a non-zero right socle. Therefore $\End(_Tl.ann_U(I_{\gamma+1}))\cong R/I_{\gamma+1}$ by the previous 
proposition.
\end{proof}

\begin{remark}\label{rem}{\rm If $\alpha$ is a non-limit ordinal, the topology on $R$ with basis of neighborwoods of $0$ the ideals $I_{\gamma+1}$, where $\gamma+1$ ranges in the set of non-limit ordinals $\gamma+1\le\alpha$, is the discrete topology. Hence, in this case also, we have that $S$ is canonically isomorphic to $\displaystyle \lim_{\longleftarrow} {}R/I_{\gamma+1}$, the completion of $R$ in the topology on $R$ with basis of neighborwoods of $0$ the ideals $I_{\gamma+1}$, where $\gamma+1$ ranges in the set of non-limit ordinals $\gamma+1\le\alpha$.}\end{remark}

\begin{remark}\label{rem'}{\rm For any right noetherian right chain ring $R$, $_TU$ turns out to be uniserial. The ideals $I_{\gamma+1}$ are the annihilators of the cyclic $T$-submodules of $_TU$. Therefore the topology on the completion $S=\displaystyle \lim_{\longleftarrow} {}R/I_{\gamma+1}$ coincides with the finite topology on $S=\Hom({}_TU, {}_TU)$.}\end{remark}

{\em Now we pass to study the endomorphism ring $S$ of any artinian uniserial module $U_R$ over an arbitrary ring $R$. }

\begin{theorem} Let $U_R$ be a non-zero artinian uniserial module. Then 
$S:=\End(U_R)$ is a local ring, its maximal ideal is $J(S)=\{\,f\in S\mid f(\soc(U_R))=0\,\}$, and $ _SU$ is a $\Sigma$-pure-injective left $S$-module.\end{theorem}

\begin{proof} Every injective endomorphism of an artinian module $U_R$ is an automorphism. Therefore the set of all endomorphisms of $U_R$ that are not invertible in $S:=\End(U_R)$ consists of all endomorphisms of $U_R$ with non-zero kernel. Since $U_R$ is artinian, the smallest non-zero submodule of $U_R$ is $\soc(U_R)$. Hence the set of all endomorphisms of $U_R$ that are not invertible in $S:=\End(U_R)$ is $\{\,f\in S\mid f(\soc(U_R))=0\,\}$. This is clearly a two-sided ideal of $S$. Hence it is the Jacobson radical of the local ring $S$. Finally, the left $S$-module $ _SU$ is  $\Sigma$-pure-injective by \cite[Ex.~2.2]{Zimmermann}.\end{proof}

As we mentioned in Section \ref{artinianuniserial}, if
 $U_R$ is an artinian uniserial module, then $\Cal L(U_R)\cong \alpha + 1$ for some ordinal $\alpha$.  
 Let $$U_0=0<U_1=\soc(U_R)<U_2<\dots<U_{\alpha}=U_R$$ be all the submodules of $U_R$. Here $U_1=\soc(U_R)$ is a simple right $R$-module. More generally, $U_{\beta+1}/U_\beta=\soc(U_R/U_\beta)$ is a simple right $R$-module for every ordinal $\beta<\alpha$. All the submodules $U_\beta$ of $U_R$ are fully invariant submodules of $U_R$, i.e., all the submodules $U_\beta$ of $U_R$ are submodules of the left module $_SU$. Thus $\Cal L(U_R)\subseteq \Cal L(_SU)$. The ordinal $\alpha$ is not a limit ordinal if and only if $U_R$ is cyclic, if and only if $U_R\cong R_R/I$ for some right ideal $I$ of $R$. In this case, $S\cong E_I/I$, where $E_I\subseteq R$ is the idealizer of $I$.

\begin{Lemma} For every finite ordinal $n<\alpha $, the left $S$-module $U_{n+1}/U_n$ is semisimple, that is, $U_{n+1}/U_n$ is a left vector space over the division ring $S/J(S)$. In particular, $_SU_\omega={}_S\soc_\omega(U_R)$ is a semiartinian left $S$-module.\end{Lemma}

\begin{proof}In order to prove that $U_{n+1}/U_n$ is semisimple as a left $S$-module, it suffices to show that $J(S)(U_{n+1}/U_n)=0$, that is, that $J(S)U_{n+1}\subseteq U_n$. Equivalently, it is sufficient to prove that if $f$ is an endomorphism of $U_R$ whose kernel contains $U_1$, then $f$ maps $U_{n+1}$ into $U_n$. This is trivial.\end{proof}

\begin{proposition} If $U_R$ is a uniserial artinian right $R$-module such that any non-zero element 
of $S$ is surjective, for example 
 if $\Cal L(U_R)\cong \omega+1$, then 
$S$  is a left noetherian left chain ring.\end{proposition}

\begin{proof}  Since $U_R$ is uniserial, any two non-zero principal left ideals of $S$ are comparable \cite[Proposition~1.2(b)]{AM}. Hence $S$ is a left chain ring. Suppose $S$ not left noetherian. Then there is a strictly ascending chain $I_0\subset I_1\subset I_2\subset \dots$ of left ideals of $S$. But $S$ is left chain, so that there is a strictly ascending chain $Sf_0\subset Sf_1\subset Sf_2\subset \dots$ of principal left ideals of $S$. By \cite[Proposition~1.2(b)]{AM} again, the chain $\ker(f_0)\supset \ker(f_1)\supset \ker(f_2)\supset \dots$ of submodules of $U_R$ is strictly descending. But $\Cal L(U_R)\cong \alpha$ implies $U_R$ artinian, a contradiction.\end{proof}

We conclude with an example of a uniserial artinian module of length $\omega 2+ 1$ whose endomorphism ring $S$ is a  chain domain. 

\begin{example}\label{Alberto's example}
{\rm
Let $\Z_ p$ be the localization of $\Z$ at a maximal ideal $(p)$, $p$
 a fixed prime number. Let $A:=\Q^{(\N)}=\bigoplus_{n\ge 0}\Q a_n$ be a vector space of countable dimension with basis $\{\,a_n\mid n\ge0\,\}$ over the field of fractions $\Q$ of $\Z_p.$
 Thus $A$ is a $\Z_p$-module containing $B:=\Z_p a_0$ as a cyclic $\Z_ p$-submodule.
 Our uniserial module will be $U:=A/B,$ as follows. Let $R$ be the subring of $\End_{\Z_p}(U)$
 consisting of all the endomorphisms of $U_{\Z_p}$ induced by the endomorphisms
  $f\in  \End_{\Z_p}(A) $ such that $f(a_0)\in B$ and $f(a_n)\in \bigoplus_{i =0}^n \Q a_i$ for all
$n\ge 1$.
Then $ _RU$ is a uniserial module of length $\omega 2+1$.

To see that $ _RU$  is uniserial, take  $u, v \in U$ non-zero. For every non-zero $u\in U$, write $u$ in the form $u=\sum_{n\ge 0}q_n a_n +B$ with the coefficients $q_n\in \Q$ almost all zero, and let $\delta(u)$ be the greatest of the indices $n\ge 0$ with $q_n\ne 0$. Without loss of generality, we can suppose $\delta(u)\ge\delta(v)$. Let us prove that $Ru\supseteq Rv$, that is, that $v\in Ru$. Assume that $\delta(u)\ge1$.  The canonical projection of $A$ onto its direct summand $\Q a_{\delta(u)} $ multiplied by $q_{\delta(u)}^{-1} $ is an endomorphism of $A$ that maps $\sum_{n\ge 0}q_n a_n$ to $a_{\delta(u)}$.
Hence this endomorphism of $A$ induces an endomorphism of $U_{Z_p}$, which is an element $r\in R$ such that
 $ru=a_{\delta(u)} + B$.
Now $\delta(u)\ge\delta(v)$ implies that $v$ can be written in the form $v=\sum_{n=0}^{\delta(u)}q'_n a_n +B$. Hence, for every $n=0,1,2,\dots, \delta(u)$, there is an endomorphism of the vector space $A$ that sends $a_{\delta(u)}$ to $q'_na_n$ and all the other elements $a_i$ with $i\ne\delta(u)$ to $0$. This endomorphism of $A$ induces an endomorphism of $U_{Z_p}$, which is an element $r_n\in R$ such that $r_na_{\delta(u)}+B=q'_na_n+B$. It follows that $(\sum_nr_nr)u=\sum_nq'_na_n+B=v$, as desired. Finally, if $\delta(u)=0$, then $\delta(v)=0$, so $u$ and $v$ are in the uniserial $\Z_p$-module $\Q a_0/\Z_p a_0$. But $\Z_p$ is contained in the center of $R$, so that clearly either $Ru\supseteq Rv$ or $Rv\supseteq Ru$.

The proper $R$-submodules of $_RU$ are the subgroups $U_n$ of the Pr\"ufer group $\Q a_0/B$ with $p^n$ elements ($n\ge 0$), and the $R$-submodules $U_{\omega+n}:=\bigoplus_{i= 0}^n\Q a_i/B$ ($n\ge 0$).
Notice that the first  $\omega$
 composition factors $U_{n+1}/U_n$ of $U_R$ are all cyclic groups of order $p$ as abelian groups, and are all isomorphic simple left $R$-modules.
The second
$\omega $ composition factors $U_{\omega+n+1}/U_{\omega+n}$ of $U_R$ are all pair-wise non-isomorphic simple left $R$-modules, though they are all isomorphic, as abelian groups, to
$ \Q.$ Thus endomorphisms of $ _RU$ are only multiplications by elements of $\Z_p,$
 which is in the center of $\End_{\Z_p}(U)$. It follows that $\End(_RU) \cong \Z_ p.$
}
\end{example} 

We conclude with an example of artinian uniserial module of length $\omega + 2$. 

\begin{example}
{\rm 
Consider the abelian group $U:=\Q\oplus\Z(p^\infty)$, and let $R$
 be the endomorphism ring of the abelian group $U$.
 Then $ _RU$ is a left $R$-module. Every element $(q,x)\in \Q\oplus\Z(p^\infty)$ with $q\ne 0$ generates 
$_RU$. It follows that the submodules of $_RU$ are the subgroup ann$_U(p^n)$, for
$ n\ge0$, $\Z(p^\infty)$ and $U$. Thus $ _RU$ is a uniserial artinian module, its lattice of submodules is isomorphic to $\omega+2$. Multiplication by $p^n$  is an epi endomorphism of $_RU$ with kernel ann$_U(p^n)$. Hence
 $U_e=\Z(p^\infty)$. 
}
\end{example}

\end{document}